 \documentclass[smallextended]{amsart}
\usepackage[utf8]{inputenc}

\usepackage{comment}
\usepackage{geometry}
\usepackage{amsmath}
\usepackage{amsfonts}
\usepackage{amssymb}
\usepackage{enumitem}
\usepackage{fancyhdr}
\usepackage{float}
\usepackage{graphicx}
\usepackage[arrowdel]{physics}
\let \div \relax

\usepackage{subcaption}
\usepackage{tkz-euclide}
\usepackage{MyPackage}

\colorlet{myRed}{red} 
\colorlet{myRed}{black} 

\newtheorem{theorem}{Theorem}[section]

\newtheorem{defin}[theorem]{Definition}
\newtheorem{lemma}[theorem]{Lemma}
\newtheorem{corollary}[theorem]{Corollary}
\newtheorem{rem}[theorem]{Remark}

\DeclareMathOperator{\E}{\mathbb{E}}

\title[Homogenization of full NSF in random perforation]
{Homogenization of the full compressible Navier-Stokes-Fourier system in\\ randomly perforated domains
}
\date{}
\author{Florian Oschmann}
\begin{document}
\maketitle


\begin{abstract}
We consider the homogenization of the compressible Navier-Stokes-Fourier equations in a randomly perforated domain in $\R^3$. Assuming that the particle size scales like $\e^\alpha$, where $\e>0$ is their mutual distance and $\alpha>3$, we show that in the limit $\e\to 0$, the velocity, density, and temperature converge to a solution of the same system. We follow the methods of Lu and Pokorn\'{y} [\text{https://doi.org/10.1016/j.jde.2020.10.032}], where they considered the full system in periodically perforated domains.
\end{abstract}

\section{Introduction}\label{sec:Introd}
We consider a bounded smooth domain $D\subset\R^3$ which for $\eps > 0$ is perforated by random balls $B_{\eps^{\alpha} r_i}(\eps z_i)$ with $\alpha > 3$, and show that solutions to the compressible Navier-Stokes-Fourier equations in this domain converge as $\e\to 0$ to a solution of the same system of equations in $D$.

There is a vast of literature concerning the homogenization of fluid flows in perforated domains. We will just cite a few. For incompressible fluids and a periodic perforation, Allaire found in \cite{Allaire90a} and \cite{Allaire90b} that, concerning the ratios of particle size and distance, there are mainly three regimes of particle sizes $\e^\alpha$, where $\alpha\geq 1$. Heuristically, if the particles are large, the velocity will slow down and finally stop. This phenomenon occurs if (in three dimensions) $\alpha\in [1,3)$ and gives rise to Darcy's law. When the particles are too small, i.e., $\alpha>3$, they should not affect the fluid, yielding that in the limit, the fluid motion is still governed by the Stokes or Navier-Stokes equations. The third regime is the so-called critical case $\alpha=3$, where the particles are large enough to put some friction on the fluid, but not too large to stop the flow. For incompressible fluids, the non-critical cases $\alpha\in (1,3)$ and $\alpha>3$ were considered in \cite{Allaire90b}, where \cite{Allaire90a} dealt with the critical case $\alpha=3$. The case $\alpha=1$ was treated in \cite{Allaire89}. In all the aforementioned literature, the proofs were given by means of suitable oscillating test functions, first introduced by Tartar in \cite{Tartar1980} and later adopted by Cioranescu and Murat in \cite{CioranescuMurat82} for the Poisson equation. The results obtained by Cioranescu and Murat and also those of Allaire can further be generalized to the case of random distributions and random radii $r_i\e^\alpha$, $r_i\geq 0$. This was done for the critical case $\alpha=3$ by Giunti, Höfer, and Vel\'{a}zquez for the Poisson equation in \cite{GiuntiHoeferVelazquez18} and by Giunti and Höfer for the Stokes equations in \cite{GiuntiHoefer19}, where they recovered Brinkman's law as in the periodic situation. The case $\alpha\in (1,3)$ was recently treated by Giunti in \cite{Giunti21}, where they recovered Darcy's law.\\

Unlike as for incompressible fluids, the homogenization theory for compressible fluids is rather sparse and focuses mainly on deterministic radii $\e^\alpha$ and a periodic distribution of holes. Masmoudi considered in \cite{Masmoudi02} the case $\alpha=1$ of large particles, giving rise to Darcy's law. For large particles with $\alpha\in (1,3)$, Darcy's law was just recently treated in \cite{SchwarzacherDarcy} for a low Mach number limit. Their methods can also be used to treat the critical case $\alpha=3$ \cite{BellaOschmann2021b}. The case of small particles ($\alpha>3$) was treated in \cite{DieningFeireislLu,FeireislLu,LuSchwarzacher} for different growing conditions on the pressure. Random perforations in the spirit of \cite{GiuntiHoefer19} for small particles were considered by Bella and the author in \cite{BellaOschmann2021a}, where in the limit, the equations remain unchanged as in the periodic case.\\

To the best of the author's knowledge, there are only two works dealing with homogenization of the full compressible Navier-Stokes-Fourier system in perforated domains, both assuming periodic distribution of the holes and deterministic radii. The article of Feireisl, Novotn\'{y} and Takahashi \cite{FeireislNovotnyTakahashi10} treats the case where the radii of the obstacles are proportional to their mutual distance. They showed that, after a proper rescaling of the velocity and suitable extensions of the density an temperature, the solutions to the compressible Navier-Stokes-Fourier equations converge to the solution of a Darcy-type law in the limiting domain. The second work is the one of Lu and Pokorn\'{y} \cite{LuPokorny21}, which focuses on the case the radii scale like $\e^\alpha$, $\alpha>3$, where $\e>0$ is the mutual distance between holes. Our methods for the case of randomly distributed holes with random radii $r_i\e^\alpha$, $r_i\geq 0$, $\alpha>3$, are therefore based on their work.\\

To obtain uniform bounds with respect to $\e$ for the solution functions, a key ingredient is the notion of the so-called Bogovski\u{\i} operator $\B_\e$ in the domain $D_\e$, which can be seen as an inverse of the divergence. Such an operator was first studied in \cite{Bogovskii80} and is known to exist for any Lipschitz domain and satisfies the norm bound $\|\B_\e\|\leq C$. However, the constant $C$ depends on the Lipschitz character of the domain $D_\e$, which is unbounded as $\e\to 0$. The key point is to develop uniform bounds for $\B_\e$ as $\e\to 0$. In the case of periodically perforated domains with deterministic radii $\e^\alpha$, $\alpha\geq 1$, this was done in \cite{DieningFeireislLu,FeireislLu,LuSchwarzacher} and recently generalized to the case of random distributions, random radii and $\alpha>2$ in \cite{BellaOschmann2021a}. We will use this Bogovski\u{\i} operator for the random case in order to generalize the results of \cite{LuPokorny21}.\\

{\bf Notation:} Throughout the whole paper, $\omega\in\Omega$, where $(\Omega,\mathcal{F},\P)$ is a suitable probability space for the marked Poisson point process as introduced in Section \ref{sec:Model} below. We further denote by $|S|$ the Lebesgue measure of a measurable set $S\subset\R^3$. We write $a\lesssim b$ whenever there is a constant $C>0$ that does not depend on $\e, a$ and $b$ such that $a\leq C\, b$. The constant $C$ might change its value whenever it occurs. The Frobenius scalar product of two matrices $A,B\in\R^{3\times 3}$ is denoted by $A:B:=\sum_{1\leq i,j\leq 3} A_{ij}B_{ij}$. Further, we use the standard notation for Lebesgue and Sobolev spaces, where we denote this spaces even for vector- or matrix valued functions as in scalar case, e.g., $L^p(D)$ instead of $L^p(D;\R^3)$.\\

{\bf Organization of the paper:} The paper is organized as follows:\\
In Section \ref{sec:Model}, we give a precise definition of the perforated domain $D_\e$ and state our main results for the steady Navier-Stokes-Fourier equations. In Section \ref{sect:ingr}, we establish uniform bounds for the velocity and density. Section \ref{sec:extension} is devoted to extend the temperature in a suitable way to the whole domain $D$, to give uniform bounds for it and to establish a trace estimate on the boundary of holes. In Section \ref{sec:fixedDom}, we show how to pass to the limit $\e\to 0$ and obtain the equations in the limiting domain.

\section{Setting and the Main Results}\label{sec:Model}
In this section we define the perforated domain, formulate the Navier-Stokes-Fourier equations governing the fluid motion, and state the main results. We start with the definition of the perforated domain.
\subsection{The perforated domain}\label{sec:domain}
Let $D \subset \R^3$ be a bounded domain with a $C^2$ boundary. For rescaling arguments, we assume $0\in D$. We model the perforation of $D$ using the Poisson point process, though the arguments can be easily generalized to a larger class of point processes. For an intensity parameter $\lambda > 0$, the Poisson point process is defined as a \emph{random} collection of points $\Phi = \{z_j\}$ in $\R^3$ characterized by the following two properties:
\begin{itemize}
 \item for any two measurable and \emph{disjoint} sets $S_1, S_2 \subset \R^3$, the random variables $S_1 \cap \Phi$ and $S_2 \cap \Phi$ are independent;
 \item for any measurable set $S \subset \R^3$ and $k \in \mathbb{N}$ holds $\mathbb P(N(S) = k) = \frac{(\lambda |S|)^ke^{-\lambda |S|}}{k!}$,
\end{itemize}
where $N(S)=\# (S\cap\Phi)$ counts the number of points $z_j\in S$ and $|S|$ denotes measure of $S$. In addition to the random locations of the balls, modeled by the above Poisson point process, we also assume the balls have random size. For that, let $\mathcal{R} = \{r_i\} \subset [0,\infty)$ be another random process of independent identically distributed random variables with finite moment bound
\begin{equation}\label{Conditionm}
 \E(r_i^{m_r}) < \infty \text{ for some $m_r>0$,}
\end{equation}
and which are independent of $\Phi$. In other words, to each point $z_j \in \Phi$ (center of a ball) we associate also a radius of the ball $r_j \in [0,\infty)$. 
The random process $(\Phi,\mathcal{R})$ on $\R^3 \times \R_+$ is called \emph{marked Poisson point process}, and can be viewed as a random variable $\omega \in \Omega \mapsto (\Phi(\omega),\mathcal{R}(\omega))$, defined on an abstract probability space $(\Omega,\mathcal{F},\mathbb{P})$. 

To define the perforated domain $D_\eps$, for $\alpha > 3$ and $\eps > 0$ we set 
\begin{equation}\label{def:Domain}
 \Phi^{\eps}(D) := \left\{ z \in \Phi \cap \frac1\eps D: \dist(\eps z,\partial D) > \eps\right\}, \quad D_\eps := D \setminus \bigcup_{z_j \in \Phi^{\eps}(D)} B_{\eps^{\alpha} r_j}(\eps z_j).
\end{equation}
To simplify the exposition and to avoid the need to analyze behavior near the boundary, we only removed those balls from $D$ which are not too close to the boundary $\partial D$. This is also a common assumption in the periodic situation, see, e.g., \cite[relation $(1.3)$]{FeireislLu}.

The exact range for the moment bound $m_r$ in \eqref{Conditionm} will be specified later on; we require at least $m_r>3/(\alpha-2)$. As shown in \cite[Theorem 3.1]{BellaOschmann2021a}, we then have the following result, which we state in form of a lemma:
\begin{lemma}\label{lem:MainProp}
Let $(\Phi,\mathcal{R})=(\{z_i\},\{r_i\})$ be a marked Poisson point process as defined above and $D_\e$ be as in \eqref{def:Domain}. Let $\alpha>2$, $m_r>3/(\alpha-2)$, $0<\delta<\alpha-1-\frac{3}{m_r}$, $\kappa\in (\max(1,\delta),\alpha-1-\frac{3}{m_r})$, and $\tau\geq 1$. Then there exists an almost surely positive random variable $\e_0(\omega)$ such that
for every $0<\e\leq \e_0$ holds \label{veta31jedna}
\begin{equation*}
\max_{z_i\in\Phi^\e(D)} \tau \e^\alpha r_i\leq \e^{1+\kappa} 
\end{equation*}
and for every $z_i,z_j\in\Phi^\e(D),\, z_i\neq z_j$, 
\begin{equation*}
\B_{\tau\e^{1+\kappa}}(\e z_i)\cap B_{\tau\e^{1+\kappa}}(\e z_j)=\emptyset.
\end{equation*}

\end{lemma}

\subsection{The Navier-Stokes-Fourier system}
We consider the stationary compressible Navier-Stokes-Fourier equations in perforated domains $D_\eps$, which describe the steady motion of a compressible and heat conducting Newtonian fluid. For $\eps > 0$, the unknown density $\rho_\eps: D_\e\to[0,\infty)$, velocity $\vb u_\eps:D_\e\to\R^3$, and temperature $\vartheta_\e:D_\e\to (0,\infty)$ of a viscous compressible fluid are described by
\begin{align}
\div(\rho_\e \vb u_\e)&=0 & \text{ in } D_\e,\label{eq:conteq}\\
\div(\rho_\e \vb u_\e\otimes \vb u_\e) + \nabla p(\rho_\e,\vartheta_\e) &= \div\mathbb{S}(\vartheta_\e,\nabla \vb u_\e) + \rho_\e \vb f+\vb g &\text{ in } D_\e,\label{eq:momentum}\\
\div(\rho_\e E\vb u_\e+p(\vartheta_\e,\rho_\e)\vb u_\e-\mathbb{S}(\vartheta_\e,\vb u_\e)\vb u_\e+\vb q_\e)&=(\rho_\e\vb f+\vb g)\cdot\vb u_\e &\text{ in } D_\e,\label{eq:energyBalance}
\end{align}
where $\mathbb{S}$ denotes the Newtonian viscous stress tensor of the form
\begin{align}\label{eq:StressTensor}
\mathbb{S}(\vartheta,\nabla \vb u)=\mu(\vartheta) \bigg(\nabla \vb u+\nabla^T \vb u-\frac23 \div(\vb u)\mathbb{I}\bigg)+\eta(\vartheta)\div(\vb u)\mathbb{I}.
\end{align}
Here we assume the viscosity coefficients $\mu(\cdot),\eta(\cdot)$ being continuous functions on $(0,\infty)$, $\mu(\cdot)$ is moreover Lipschitz continuous, and
\begin{align}\label{eq:ViscosityCoeff}
C_1(1+\vartheta)\leq \mu(\vartheta)\leq C_2(1+\vartheta),\quad 0\leq \eta(\vartheta)\leq C_2(1+\vartheta).
\end{align}
We further impose boundary conditions on $\d D_\e$ as
\begin{align}\label{bdryCond}
\begin{split}
\vb u_\e&=0,\\
\vb q_\e\cdot \vb n&=L(\vartheta_\e-\vartheta_0),
\end{split}
\end{align}
where $\vartheta_0\geq T_0>0$ is a prescribed temperature distribution in $D$ and $L>0$ a given constant, and fix the total mass by
\begin{align}\label{eq:fixedMass}
\int_{D_\e}\rho_\e=M>0,
\end{align}
where $M>0$ is independent of $\e$.

For the constitutive law of the pressure, we assume
\begin{align}\label{constitutiveLaw}
p(\vartheta,\rho)=a\rho^\gamma+c_v(\gamma-1)\rho\vartheta,
\end{align}
where $a>0$, $\gamma>2$ is the adiabatic exponent and $c_v>0$ is the specific heat capacity. Note that the thermodynamic part of the pressure is just the ideal gas law $pV=R\vartheta$ with $V=1/\rho$ and universal gas constant $R=c_p-c_v=c_v(\gamma-1)>0$. The heat flux is governed by Fourier's law
\begin{align}\label{FourierLaw}
\vb q(\vartheta,\nabla\vartheta)=-\kappa(\vartheta)\nabla\vartheta,
\end{align}
where we assume the heat conductivity $\kappa$ to satisfy
\begin{align}\label{HeatConductivity}
C_3(1+\vartheta^{m_\vartheta})\leq \kappa(\vartheta)\leq C_4(1+\vartheta^{m_\vartheta})
\end{align}
for some $m_\vartheta>2$. The total energy is given by
\begin{align}\label{totalEnergy}
E=e+\frac12 |\vb u|^2,
\end{align}
where the specific energy $e$ satisfies Gibb's relation
\begin{align}\label{GibbsRelation}
\frac1\vartheta \bigg(De+p(\vartheta,\rho)D\bigg(\frac1\rho\bigg)\bigg)=Ds(\rho,\vartheta).
\end{align}
Assuming the entropy for an ideal fluid as $s(\rho,\vartheta)=c_v\log\big(\frac{\vartheta}{\rho^{\gamma-1}}\big)$, this leads to
\begin{align}\label{specificEnergy}
e(\rho,\vartheta)=c_v\vartheta+\frac{\rho^{\gamma-1}}{\gamma-1}.
\end{align}
Further, the entropy $s$ fulfils formally the balance of entropy
\begin{align*}
\div\bigg(\rho s\vb u+\frac{\vb q}{\vartheta}\bigg)=\frac{\mathbb{S}:\nabla\vb u}{\vartheta}-\frac{\vb q\cdot\nabla\vartheta}{\vartheta^2}.
\end{align*}
Finally, we assume the external forces $\vb f,\vb g\in L^\infty(\R^3)$.

Since the existence of classical solutions to~\eqref{eq:conteq}--\eqref{eq:energyBalance} is known only if the data are in a certain sense ``small'' (see, e.g., \cite{DaVeiga1987,PiaseckiPokorny2014} and the references therein), we will work with weak solutions, which are known to exist under even weaker assumptions of $m_\vartheta$ and $\gamma$ as made above.

\subsection{Weak formulation and weak solutions}
Here we state the weak formulation of the problem in $D_\e$. To simplify notation, we will identify a function with $D_\e$ as its domain of definition with its zero extension to the whole of $\R^3$.

First, the weak formulation of the continuity equation reads
\begin{align}\label{eq:weakContinuity}
\int_{\R^3} \rho_\e\vb u_\e\cdot\nabla\psi=0
\end{align} 
for all $\psi\in C_c^1(\R^3)$. We will moreover work with a renormalized version of this, that is,
\begin{align}\label{eq:renormalized}
\int_{\R^3} b(\rho_\e)\vb u_\e\cdot\nabla\psi+(b(\rho_\e)-\rho_\e b^\prime(\rho_\e))\div(\vb u_\e)\psi=0
\end{align}
for any $\psi\in C_c^1(\R^3)$ and any $b\in C^1([0,\infty))$ such that $b^\prime\in C_0([0,\infty))$. We remark that the assumptions on $b$ can be relaxed, see, for instance, \cite{DiPernaLions}.

The weak formulation of the momentum equation reads
\begin{align}\label{eq:weakMomentum}
\int_{D_\e} p(\rho_\e,\vartheta_\e)\div \phi+(\rho_\e \vb u_\e\otimes \vb u_\e):\nabla\phi-\mathbb{S}(\vartheta_\e,\nabla \vb u_\e):\nabla\phi+(\rho_\e \vb f+\vb g)\cdot\phi=0
\end{align}
for any $\phi\in C_c^1(D_\e;\R^3)$.

The weak formulation of the energy balance reads
\begin{align}\label{eq:weakEnergy}
-\int_{D_\e}\bigg(\rho_\e E\vb u_\e+p(\vartheta_\e,\rho_\e)\vb u_\e-\mathbb{S}(\vartheta_\e,\nabla\vb u_\e)\vb u_\e+\vb q_\e\bigg)\cdot\nabla\psi+\int_{\d D_\e}L(\vartheta_\e-\vartheta_0)\psi=\int_{D_\e}(\rho_\e\vb f+\vb g)\cdot\vb u_\e \psi
\end{align}
for all $\psi\in C^1(\overline{D_\e})$. Farther, we also have the energy inequality
\begin{align}\label{eq:entropyInequ}
\int_{D_\e}\bigg(\frac{\mathbb{S}(\vartheta_\e,\nabla\vb u_\e):\nabla\vb u_\e}{\vartheta_\e}-\frac{\vb q_\e\cdot\nabla\vartheta_\e}{\vartheta_\e^2}\bigg)\psi+\int_{\d D_\e}\frac{L\vartheta_0}{\vartheta_\e}\psi\leq -\int_{D_\e} \bigg(\rho_\e s(\vartheta_\e,\rho_\e)\vb u_\e+\frac{\vb q_\e}{\vartheta_\e}\bigg)\cdot\nabla\psi+L\int_{\d D_\e}\psi
\end{align}
for all $\psi\in C^1(\overline{D_\e})$ with $\psi\geq 0$.

\begin{defin}\label{defin:weakSol}
The triple $(\rho,\vb u, \vartheta)$ is said to be a \emph{renormalized weak entropy solution} to problem \eqref{eq:conteq}--\eqref{specificEnergy} if $\rho\geq 0,\vartheta>0$ a.e.~in $D_\e$, $\rho\in L^\gamma(D_\e)$, $u\in H_0^1(D_\e;\R^3)$, $\vartheta^{m_\vartheta/2}$ and $\log\vartheta\in H^1(D_\e)$ such that $\rho |\vb u|^3$, $|\mathbb{S}(\vartheta,\nabla\vb u)\vb u|$ and $p(\vartheta,\rho)|\vb u|\in L^1(D_\e)$, and the relations \eqref{eq:weakContinuity}--\eqref{eq:entropyInequ} are fulfiled.
\end{defin}

For $\e>0$ fixed, the existence of weak solutions is guaranteed by the following result, see \cite{NovotnyPokorny2011} for details:
\begin{theorem}
Let $\vb f,\vb g\in L^\infty(\R^3)$, $\vartheta_0\in L^1(\d D_\e)$, $\vartheta_0\geq T_0>0$ a.e.~on $\d D_\e$, $L>0$ and $M>0$. Let $\gamma>\frac53$ and $m_\vartheta>1$. Then there exists a renormalized weak entropy solution $(\rho,\vb u,\vartheta)$ to problem \eqref{eq:conteq}--\eqref{specificEnergy} in the sense of Definition \ref{defin:weakSol}.
\end{theorem}

\subsection{Main result}
Before stating our main result concerning the Navier-Stokes-Fourier system \eqref{eq:conteq}--\eqref{eq:energyBalance}, we want to state a result on the existence and boundedness of an inverse to the divergence operator, which is proven in \cite[Theorem 2.1]{BellaOschmann2021a}.

\begin{lemma}\label{MainBog}
Let $(\Phi,\mathcal{R})$ be a marked Poisson point process as defined in Section \ref{sec:domain} and $D_\e$ be as in \eqref{def:Domain}. Let $\alpha>2$ and $m_r>3/(\alpha-2)$. Then, for all $1<q<3$ satisfying
\begin{align}\label{Conditionq}
\alpha-\frac{3}{m_r}>\frac{3}{3-q},
\end{align}
there exists an almost surely positive random variable $\e_0(\omega)$ such that for all ${0<\e\leq \e_0}$ there exists a bounded linear operator
\begin{align*}
\B_\e:L^q(D_\e)/\R\to W_0^{1,q}(D_\e;\R^3)
\end{align*}
such that for all $f\in L^q(D_\e)$ with $\int_{D_\eps} f = 0$
\begin{align*}
\div (\B_\e(f))=f \text{ in } D_\e,\quad \|\B_\e(f)\|_{W_0^{1,q}(D_\e)}\leq C  \, \|f\|_{L^q(D_\e)},
\end{align*}
where the constant $C>0$ is independent of $\omega$ and $\e$.  
\end{lemma}

We are now in the position to state our main result, which generalizes \cite[Theorem 2.2]{LuPokorny21}:
\begin{theorem}\label{thm:main}
Let $(\Phi,\mathcal{R})=(\{z_i\},\{r_i\})$ and $D_\e$ be defined as in Section \ref{sec:domain}. Let $\vb f, \vb g\in L^\infty(\R^3)$, $M>0$, $L>0$ and $\vartheta_0\geq T_0>0$ in $D$ be defined such that it has finite $L^q$-norm over all smooth two-dimensional surfaces with finite surface area contained in $D$ for some $q>1$. Let $(\rho_\e,\vb u_\e,\vartheta_\e)$ be a sequence of renormalized weak entropy solutions to problem \eqref{eq:conteq}--\eqref{specificEnergy}, extended in a suitable way to the whole domain $D$ as shown in Section \ref{sec:extension} below. Let $\alpha>3$, $\gamma>2$, $m_\vartheta>2$ and $m_r>\max\{3/(\alpha-3),3\}$ satisfy the relation
\begin{align}\label{conditionAlpha}
\alpha-\frac{3}{m_r}>\max\left\lbrace\frac{2\gamma-3}{\gamma-2},\frac{3m_\vartheta-2}{m_\vartheta-2}\right\rbrace.
\end{align}

Then, there exists an almost surely positive random variable $\e_0(\omega)$ such that for all $0<\e\leq \e_0$ there hold the uniform bounds
\begin{align*}
\|\rho_\e\|_{L^{\gamma+\Theta}(D)}+\|\vb u_\e\|_{H_0^1(D)}+\|\vartheta_\e\|_{H^1(D)\cap L^{3m_\vartheta}(D)}\leq C,
\end{align*}
where $\Theta:=\min\{2\gamma-3,\gamma\frac{3m_\vartheta-2}{3m_\vartheta+2}\}$. Moreover, the corresponding weak limit as $\e\to 0$ is a renormalized weak solution to problem \eqref{eq:conteq}--\eqref{specificEnergy} in the limit domain $D$, i.e., $\rho\geq 0$ and $\vartheta>0$ a.e.~in $D$ and the equations \eqref{eq:weakContinuity}-\eqref{eq:weakEnergy} are fulfiled.
\end{theorem}

\begin{rem}
Note that we do not know whether the energy inequality \eqref{eq:entropyInequ} is fulfiled in the limit. This fact is even not known for the case of constant radii and a periodic distribution of the holes, see \cite[Section 2.4]{LuPokorny21}.
\end{rem}

\begin{rem}\label{rem:renormalized}
Due to the DiPerna-Lions transport theory (see \cite{DiPernaLions}), for any smooth domain ${D\subset\R^3}$, any $r\in L^\beta(D)$ with $\beta\geq 2$ and any $\vb v\in H_0^1(D)$ such that
\begin{align*}
\div(r\vb v)=0\text{ in } \mathcal{D}^\prime(D),
\end{align*}
the couple $(r,\vb v)$, extended by zero outside $D$, satisfies the renormalized equation
\begin{align*}
\div((b(r)\vb v)+(rb^\prime(r)-b(r))\div\vb v=0 \text{ in }\mathcal{D}^\prime(\R^3),
\end{align*}
where $b\in C([0,\infty))\cap C^1((0,\infty))$ satisfies
\begin{align*}
b^\prime(s)\leq Cs^{-\lambda_0} \text{ for } s\in (0,1],\quad b^\prime(s)\leq Cs^{\lambda_1} \text{ for } s\in [1,\infty)
\end{align*}
with constants
\begin{align*}
C>0,\quad \lambda_0<1,\quad -1<\lambda_1\leq \frac\beta2-1.
\end{align*}
Thus, the renormalized continuity equation \eqref{eq:renormalized} is satisfied for any function $b$ satisfying the weaker assumptions $b\in C^1([0,\infty))$ and $b^\prime\in C_0([0,\infty))$.
\end{rem}

\section{Uniform bounds}\label{sect:ingr}
In this section, we give uniform bounds on the velocity and the density. We will always assume that the requirements on $\alpha,m_r,m_\vartheta,$ and $\gamma$ as made above hold and that the moment bound $m_r$ in \eqref{Conditionm} satisfies the additional assumption $m_r\geq 3$ in order to control the measure of the boundary $\d D_\e$ and the measure of $D_\e$ itself.

The entropy inequality \eqref{eq:entropyInequ} enables us to get several bounds on the sequence $(\rho_\e,\vb u_\e,\vartheta_\e)$ in $D_\e$. Before stating these bounds, we need the following form of the Strong Law of Large Numbers (see \cite[Lemma C.1]{GiuntiHoefer19}):
\begin{lemma}\label{lem:SLLN}
Let $d\geq 1$ and $(\Phi,\mathcal{R})=(\{z_j\},\{r_j\})$ be a marked Poisson point process on $\R^d\times\R_+$ with intensity $\lambda>0$. Assume that the marks $\{r_j\}$ are non-negative i.i.d.~random variables independent of $\Phi$ such that ${\E(r_j^{m_r}) <\infty}$ for some $m_r > 0$. Then, for every bounded set $S\subset\R^d$ which is star-shaped with respect to the origin, we have almost surely
\begin{align*}
\lim_{\e\to 0} \e^d N(\e^{-1} S)=\lambda |S|, \qquad \lim_{\e\to 0} \e^d \hspace{-0.3em} \sum_{z_j \in \eps^{-1} S} r_j^{m_r}=\lambda \E( r^{m_r})|S|.
\end{align*}
\end{lemma}

\begin{rem}\label{rmk:SLLN} 
 Assuming the boundary of the set $S$ from the previous lemma is not too large, the same argument also shows
\begin{equation}\label{eq:ergodic}
\lim_{\e\to 0} \e^d N(\e^{-1} S)=\lambda |S|, \qquad \lim_{\e\to 0} \e^d \hspace{-0.3em} \sum_{z_j \in \Phi^{\eps}(S)} r_j^{m_r}=\lambda \E( r^{m_r})|S|.
\end{equation}
In particular, it is enough that $S$ has as $D$ a $C^2$-boundary.  
\end{rem}

Together with Lemma \ref{lem:SLLN}, we obtain for $\e>0$ small enough
\begin{align}\label{measures}
\begin{split}
|\d D_\e|&=|\d D|+\bigg|\bigcup_{z_i\in\Phi^\e(D)} \d B_{\e^\alpha r_i}(\e z_i)\bigg|\leq C+C\e^{2\alpha-3}\e^3\sum_{z_i\in\Phi^\e(D)} r_i^2\leq C,\\
|D\setminus D_\e|&=\bigg|\bigcup_{z_i\in\Phi^\e(D)} B_{\e^\alpha r_i}(\e z_i)\bigg|\leq C\e^{3(\alpha-1)}\e^3\sum_{z_i\in\Phi^\e(D)} r_i^3\leq C\e^{3(\alpha-1)},
\end{split}
\end{align}
which implies $|D_\e|\to |D|$ as $\e\to 0$ and thus $|D_\e|\leq 2\, |D|$ for $\e>0$ possibly even smaller. This yields for the entropy inequality \eqref{eq:entropyInequ} with $\psi\equiv 1$
\begin{align*}
\int_{D_\e}\frac{\mathbb{S}(\vartheta_\e,\nabla\vb u_\e):\vb u_\e}{\vartheta_\e}+\frac{(1+\vartheta_\e^{m_\vartheta})|\nabla\vartheta_\e|^2}{\vartheta_\e^2}+\int_{\d D_\e}\frac{L\vartheta_0}{\vartheta_\e}\leq C |\d D_\e|\leq C.
\end{align*}
If we take also $\psi\equiv 1$ in the weak formulation of the energy balance \eqref{eq:weakEnergy}, we obtain
\begin{align*}
L\int_{\d D_\e}\vartheta_\e\leq C\bigg(1+\int_{D_\e} (\rho_\e+1)|\vb u_\e|\bigg)\leq C\big(1+(\|\rho_\e\|_{L^\frac65(D_\e)}+1)\|\vb u_\e\|_{L^6(D_\e)}\big).
\end{align*}
Hence, due to the form of the stress tensor in \eqref{eq:StressTensor} and Korn's inequality, we have
\begin{align}\label{eq:firstBounds}
\begin{split}
&\|\vb u_\e\|_{H_0^1(D_\e)}+\|\nabla\log\vartheta_\e\|_{L^2(D_\e)}+\|\nabla|\vartheta_\e|^\frac{m_\vartheta}{2}\|_{L^2(D_\e)}+\|\vartheta^{-1}_\e\|_{L^1(\d D_\e)}\leq C,\\
&\|\vartheta_\e\|_{L^1(\d D_\e)}\leq C(1+\|\rho_\e\|_{L^\frac65(D_\e)}).
\end{split}
\end{align}

Note that the bounds in \eqref{eq:firstBounds} imply, by Sobolev inequality, that the norm $\|\vartheta_\e\|_{L^{3m_\vartheta}(D_\e)}$ is controlled by $\|\rho_\e\|_{L^\frac65(D_\e)}$. However, we do not know whether $\vartheta_\e$ is \emph{uniformly} bounded. To prove this, we need some additional tools. We will do this in the next section independent of the following results. For now, we will assume that $\vartheta_\e$ is uniformly bounded in $L^{3m_\vartheta}(D_\e)$ and prove this fact later on.

To get uniform bounds on the density, we will use Lemma \ref{MainBog} and proceed similar to \cite{BellaOschmann2021a,DieningFeireislLu,LuPokorny21}.

\begin{lemma}[see \cite{LuPokorny21}, Lemma 3.1]
Under the assumptions of Lemma \ref{MainBog}, assume additionally that $\|\vartheta_\e\|_{L^{3m_\vartheta}(D_\e)}$ is uniformly bounded. Then, for $\e>0$ small enough, we have
\begin{align*}
\|\rho_\e\|_{L^{\gamma+\Theta}(D_\e)}\leq C,
\end{align*}
where $C>0$ is independent of $\e$ and
\begin{align}\label{definTheta}
\Theta:=\min\left\lbrace2\gamma-3,\gamma\frac{3m_\vartheta-2}{3m_\vartheta+2}\right\rbrace.
\end{align}
\end{lemma}
\begin{proof}
In the weak formulation of the momentum balance \eqref{eq:weakMomentum}, we will use the test function
\begin{align*}
\phi:=\B_\e\big(\rho_\e^\Theta-\langle\rho_\e^\Theta\rangle\big),\quad \langle\rho_\e^\Theta\rangle:=\frac{1}{|D_\e|}\int_{D_\e}\rho_\e^\Theta,
\end{align*}
where $\B_\e$ is the operator from Lemma \ref{MainBog}, and $\Theta$ to be determined. We then have for any $1<q<3$ satisfying \eqref{Conditionq}
\begin{align*}
\|\nabla\phi\|_{L^q(D_\e)}\leq C(q)\,\|\rho_\e^\Theta\|_{L^q(D_\e)}.
\end{align*}
Using $\phi$ as test function in \eqref{eq:weakMomentum}, we get
\begin{align*}
\int_{D_\e} p(\rho_\e,\vartheta_\e)\rho_\e^\Theta=\int_{D_\e}p(\vartheta_\e,\rho_\e)\langle\rho_\e^\Theta\rangle+\mathbb{S}(\vartheta_\e,\nabla \vb u_\e):\nabla\phi-(\rho_\e \vb u_\e\otimes \vb u_\e):\nabla\phi-(\rho_\e \vb f+\vb g)\cdot\phi.
\end{align*}
We will estimate the right hand-side term by term and start with the most restrictive terms, which will give bounds on $\Theta$. First, we take the convective term to estimate
\begin{align*}
\int_{D_\e} \big|(\rho_\e\vb u_\e\otimes\vb u_\e):\nabla\phi\big| &\leq \|\vb u_\e\|_{L^6(D_\e)}^2\|\rho_\e\|_{L^{\gamma+\Theta}(D_\e)}\|\nabla\phi\|_{L^{q_1}(D_\e)}\\
&\leq C(q_1)\,\|\vb u_\e\|_{L^6(D_\e)}^2\|\rho_\e\|_{L^{\gamma+\Theta}(D_\e)}\|\rho_\e^\Theta\|_{L^{q_1}(D_\e)}\\
&=C(q_1)\,\|\vb u_\e\|_{L^6(D_\e)}^2\|\rho_\e\|_{L^{\gamma+\Theta}(D_\e)}\|\rho_\e\|_{L^{q_1\Theta}(D_\e)}^\Theta,
\end{align*}
where $q_1$ is determined by
\begin{align*}
\frac{1}{q_1}=1-\frac26-\frac{1}{\gamma+\Theta}.
\end{align*}
In order to get as high integrability of $\rho_\e$ as possible, we choose $\Theta$ such that $q_1\Theta=\gamma+\Theta$. This together with $\gamma>2$ leads to
\begin{align*}
\Theta=\Theta_1:=2\gamma-3>1,\quad q_1=\frac{3(\gamma-1)}{2\gamma-3}\in (\frac32,3).
\end{align*}
Using Sobolev embedding and the uniform bound on $\vb u_\e$ from \eqref{eq:firstBounds} to obtain $\|\vb u_\e\|_{L^6(D_\e)}\leq C\, \|\vb u_\e\|_{H_0^1(D_\e)}\leq C$, we deduce
\begin{align*}
\int_{D_\e} \big|(\rho_\e\vb u_\e\otimes\vb u_\e):\nabla\phi\big|\leq C\|\rho_\e\|_{L^{\gamma+\Theta_1}(D_\e)}^{1+\Theta_1},
\end{align*}
where $C>0$ is independent of $\e$ and $1+\Theta_1<\gamma+\Theta_1$.

Second, we consider the diffusive term to obtain
\begin{align*}
\int_{D_\e} \big|\mathbb{S}(\vartheta_\e,\nabla \vb u_\e):\nabla\phi\big| &\leq C\,(1+\|\vartheta_\e\|_{L^{3m_\vartheta}(D_\e)})\|\nabla\vb u_\e\|_{L^2(D_\e)}\|\nabla\phi\|_{L^{q_2}(D_\e)}\\
&\leq C(q_2)\,\|\nabla\vb u_\e\|_{L^2(D_\e)}\|\rho_\e^\Theta\|_{L^{q_2}(D_\e)}\\
&=C(q_2)\,\|\nabla\vb u_\e\|_{L^2(D_\e)}\|\rho_\e\|^\Theta_{L^{q_2\Theta}(D_\e)},
\end{align*}
where we set (recall $m_\vartheta>2$)
\begin{align*}
q_2:=\frac{6m_\vartheta}{3m_\vartheta-2}\in (2,3).
\end{align*}
As before, we choose $\Theta$ such that $q_2\Theta=\gamma+\Theta$, which leads to
\begin{align*}
\Theta=\Theta_2:=\gamma\frac{3m_\vartheta-2}{3m_\vartheta+2}>1.
\end{align*}
This yields
\begin{align*}
\int_{D_\e} \big|\mathbb{S}(\vartheta_\e,\nabla \vb u_\e):\nabla\phi\big|\leq C \|\rho_\e\|_{L^{\gamma+\Theta}(D_\e)}^\Theta.
\end{align*}
In particular, if we set
\begin{align*}
\Theta:=\min\{\Theta_1,\Theta_2\}>1,\quad \alpha-\frac{3}{m_r}>\max\left\lbrace\frac{2\gamma-3}{\gamma-2},\frac{3m_\vartheta-2}{m_\vartheta-2}\right\rbrace>3,
\end{align*}
then $q_1$ and $q_2$ satisfy \eqref{Conditionq} and we infer
\begin{align*}
\int_{D_\e} \big|(\rho_\e\vb u_\e\otimes\vb u_\e):\nabla\phi\big|+\int_{D_\e} \big|\mathbb{S}(\vartheta_\e,\nabla \vb u_\e):\nabla\phi\big|\leq C\big(1+\|\rho_\e\|_{L^{\gamma+\Theta}(D_\e)}^{1+\Theta}\big).
\end{align*}

Since $m_\vartheta>2$, we have $\Theta\leq \gamma\frac{3m_\vartheta-2}{3m_\vartheta+2}<\gamma$, yielding $2 \Theta<\gamma+\Theta$. Thus we infer
\begin{align*}
\int_{D_\e} \big|(\rho_\e\vb f+\vb g)\cdot\phi\big|&\leq C\,(\|\rho_\e\|_{L^2(D_\e)}+1)\|\phi\|_{L^2(D_\e)}\\
&\leq C(2)\,(\|\rho_\e\|_{L^{\gamma+\Theta}(D_\e)}+1)\|\rho_\e\|_{L^{2\Theta}(D_\e)}^\Theta\\
&\leq C(2)\,(\|\rho_\e\|_{L^{\gamma+\Theta}(D_\e)}+1)\|\rho_\e\|_{L^{\gamma+\Theta}(D_\e)}^\Theta\\
&\leq C(2)\, (\|\rho_\e\|_{L^{\gamma+\Theta}(D_\e)}^{1+\Theta}+1),
\end{align*}
where in the last inequality we used
\begin{align}\label{eq:Young}
ab^\frac1p\leq b+a^{p^\prime}\quad \forall a,b\geq 0,\, \frac{1}{p}+\frac{1}{p^\prime}=1,
\end{align}
which is a consequence of Young's inequality, for $b=\|\rho_\e\|_{L^{\gamma+\Theta}(D_\e)}^{1+\Theta}$ and $p=(1+\Theta)/\Theta$.

Farther, the estimate for the pressure reads
\begin{align*}
\int_{D_\e}\big|p(\vartheta_\e,\rho_\e)\langle\rho_\e^\Theta\rangle\big|&\leq C\int_{D_\e} (\rho_\e^\gamma+\rho_\e\vartheta_\e)\langle\rho_\e^\Theta\rangle\\
&\leq C \bigg(\|\rho_\e\|_{L^\gamma(D_\e)}^\gamma+\|\rho_\e\|_{L^\frac{6}{5}(D_\e)}\|\vartheta_\e\|_{L^6(D_\e)}\bigg)\|\rho_\e\|_{L^\Theta(D_\e)}^\Theta\\
&\leq C \bigg(\|\rho_\e\|_{L^\gamma(D_\e)}^\gamma+\|\rho_\e\|_{L^\gamma(D_\e)}\bigg)\|\rho_\e\|_{L^\Theta(D_\e)}^\Theta.
\end{align*}
Here we assumed that $\vartheta_\e$ is bounded in $L^{3 m_\vartheta}(D_\e)\subset L^6(D_\e)$. Using \eqref{eq:Young} for $b=\|\rho_\e\|_{L^\gamma(D_\e)}^\gamma$ and $p=\gamma$, together with $\Theta<\gamma$, which implies $\|\rho_\e\|_{L^\Theta(D_\e)}^\Theta\leq 1+\|\rho_\e\|_{L^\gamma(D_\e)}^\gamma\|\rho_\e\|_{L^\Theta(D_\e)}^\Theta$, interpolation between the norms of $L^1(D_\e)$ and $L^{\gamma+\Theta}(D_\e)$ and that we control the $L^1$-norm of $\rho_\e$ (i.e., the total mass), we end up with
\begin{align*}
\int_{D_\e}\big|p(\vartheta_\e,\rho_\e)\langle\rho_\e^\Theta\rangle\big|&\leq C \bigg(1+\|\rho_\e\|_{L^\gamma(D_\e)}^\gamma\bigg)\|\rho_\e\|_{L^\Theta(D_\e)}^\Theta\\
&\leq C \bigg(1+\|\rho_\e\|_{L^\gamma(D_\e)}^\gamma\|\rho_\e\|_{L^\Theta(D_\e)}^\Theta\bigg)\\
&\leq C\bigg(1+\|\rho_\e\|_{L^{\gamma+\Theta}(D_\e)}^\lambda\bigg)
\end{align*}
for some $\lambda<\gamma+\Theta$.

Finally, we obtain
\begin{align*}
\|\rho_\e\|_{L^{\gamma+\Theta}(D_\e)}^{\gamma+\Theta}\leq C\bigg(1+\|\rho_\e\|_{L^{\gamma+\Theta}(D_\e)}^\lambda\bigg)\quad\text{ for some } 1<\lambda<\gamma+\Theta,
\end{align*}
which yields the uniform bound on $\rho_\e$ in $L^{\gamma+\Theta}(D_\e)$, provided $\vartheta_\e$ is uniformly bounded in $L^{3m_\vartheta}(D_\e)$.

\end{proof}

Combining the uniform estimates on $\rho_\e$ with these from \eqref{eq:firstBounds}, we obtain
\begin{gather*}
\|\vb u_\e\|_{H_0^1(D_\e)}+\|\rho_\e\|_{L^{\gamma+\Theta}(D_\e)}+\|\nabla\log\vartheta_\e\|_{L^2(D_\e)}+\|\nabla|\vartheta_\e|^\frac{m_\vartheta}{2}\|_{L^2(D_\e)}\leq C,\\
\|\vartheta_\e\|_{L^1(\d D_\e)}+\|\vartheta^{-1}_\e\|_{L^1(\d D_\e)}\leq C.
\end{gather*}
Note that these bounds are obtained by using the assumption that $\vartheta_\e$ is uniformly bounded in $L^{3m_\vartheta}(D_\e)$. This assumption will be proven in the next section.

\section{Extension of functions}\label{sec:extension}
In order to work in the fixed domain $D$ instead of the variable domain $D_\e$, we can extend the functions $\vb u_\e$ and $\rho_\e$ simply by zero, which will preserve their regularity and their norms. In particular, the extended functions are still uniformly bounded.

However, the extension of the temperature is more delicate since an extension by zero will in general not preserve its regularity. Since this extension was previously done in \cite[Section 4]{LuPokorny21}, we will not repeat the full arguments of the proofs. First recall that, by Lemma \ref{lem:MainProp} and for $\e>0$ small enough, the balls $\{B_{2\e^\alpha r_i}(\e z_i)\}_{z_i\in\Phi^\e(D)}$ are disjoint. The first lemma we need thus follows by a trivial modification of the proof from \cite[Lemma 4.1]{LuPokorny21}:
\begin{lemma}
Let $D_\e$ be defined as in \eqref{def:Domain} and let the assumptions of Lemma \ref{lem:MainProp} hold. Then there is an almost surely positive random variable $\e_0(\omega)$ such that for all $0<\e\leq\e_0$ there exists an extension operator $E_\e:H^1(D_\e)\to H^1(D)$ such that for any $\phi\in H^1(D_\e)$ and any $z_i\in\Phi^\e(D)$,
\begin{align*}
&E_\e\phi=\phi\text{ in } D_\e,\\
&\|\nabla E_\e\phi\|_{L^2(B_{\e^\alpha r_i}(\e z_i))}\leq C\|\nabla\phi\|_{L^2(B_{2\e^\alpha r_i}(\e z_i)\setminus B_{\e^\alpha r_i}(\e z_i))}
\end{align*}
and hence $\|\nabla E_\e\phi\|_{L^2(D)}\leq C\|\nabla\phi\|_{L^2(D_\e)}$. Farther, for  any $1\leq q\leq\infty$,
\begin{align*}
\| E_\e\phi\|_{L^q(B_{\e^\alpha r_i}(\e z_i))}\leq C\|\phi\|_{L^q(B_{2\e^\alpha r_i}(\e z_i)\setminus B_{e^\alpha r_i}(\e z_i))},
\end{align*}
where the constant $C>0$ is independent of $\e$ and $i$. Furthermore, there exists an operator ${\tilde{E}_\e:H^1_{\geq 0}(D_\e)\to H^1_{\geq 0}(D)}$ with the same properties as above. Here $H^1_{\geq 0}$ denotes the Sobolev space of all non-negative functions in $H^1$. In particular, one may choose $\tilde{E}_\e\phi:=\max\{0,E_\e\phi\}$.
\end{lemma}

With the help of the extension operator $\tilde{E}_\e$, we can bound the temperature uniformly w.r.t.~$\e$:
\begin{lemma}
For $\e>0$ small enough, we have $\|\tilde{E}_\e\vartheta_\e\|_{H^1(D)}+\|\tilde{E}_\e\vartheta_\e\|_{L^{3m_\vartheta}(D)}\leq C$ for some $C>0$ independent of $\e$. In particular, we have $\|\vartheta_\e\|_{H^1(D_\e)}+\|\vartheta_\e\|_{L^{3m_\vartheta}(D_\e)}\leq C$ uniformly in $\e$.
\end{lemma}

We further need to estimate the trace of $\vartheta_\e$ in $\d D_\e$. Indeed, for fixed $\e>0$, the trace of $\vartheta_\e$ belongs to $L^{2m_\vartheta}(\d D_\e)$. The next lemma enables us to control its norm in a quantitative way:
\begin{lemma}\label{lem:traceSingleBall}
Under the assumptions of Theorem \ref{thm:main}, we have for any $z_i\in\Phi^\e(D)$ and for $\e>0$ small enough
\begin{align*}
\|\vartheta_\e\|_{L^{2m_\vartheta}(\d B_i)}^{2m_\vartheta}\leq C\bigg(\|\nabla|\vartheta|^\frac{m_\vartheta}{2}\|_{L^2(2B_i\setminus B_i)}^2+\|\vartheta_\e\|_{L^{3m_\vartheta}(2B_i\setminus B_i)}^{3m_\vartheta}+\|\vartheta_\e\|_{L^{3m_\vartheta}(2B_i\setminus B_i)}^{2m_\vartheta}\bigg),
\end{align*}
where we set $B_i:=B_{\e^\alpha r_i}(\e z_i)$ and $2B_i:=B_{2\e^\alpha r_i}(\e z_i)$.
\end{lemma}

The last ingredient we need is a trace estimate for the whole boundary of the holes, which was again given in \cite[Corollary 4.1]{LuPokorny21}.
\begin{corollary}\label{cor:TraceBound}
Under the assumptions of Lemma \ref{lem:MainProp} and Theorem \ref{thm:main}, we have for any $z_i\in\Phi^\e(D)$ and for $\e>0$ small enough
\begin{align*}
\|\vartheta_\e\|_{L^{2m_\vartheta}(\cup_{z_i\in\Phi^\e(D)} \d B_{\e^\alpha r_i}(\e z_i))}\leq C\e^{-\frac{1}{2m_\vartheta}}.
\end{align*}
\end{corollary}
\begin{proof}
For $z_i\in\Phi^\e(D)$, we set again $B_i:=B_{\e^\alpha r_i}(\e z_i)$ and $2B_i:=B_{2\e^\alpha r_i}(\e z_i)$. Then, using Hölder's inequality and Lemma \ref{lem:traceSingleBall}, we get
\begin{align*}
\int_{\cup_{z_i\in\Phi^\e(D)} \d B_i} |\vartheta_\e|^{2m_\vartheta}&=\sum_{z_i\in\Phi^\e(D)}\int_{\d B_i} |\vartheta_\e|^{2m_\vartheta}\\
&\lesssim\sum_{z_i\in\Phi^\e(D)} \bigg(\int_{2B_i\setminus B_i} |\vartheta_\e|^{3m_\vartheta}\bigg)^\frac23\\
&\quad+\sum_{z_i\in\Phi^\e(D)} \int_{2B_i\setminus B_i}\big|\nabla|\vartheta_\e|^\frac{m_\vartheta}{2}\big|^2+\sum_{z_i\in\Phi^\e(D)}\int_{2B_i\setminus B_i}|\vartheta|^{3m_\vartheta}\\
&\lesssim \bigg(\sum_{z_i\in\Phi^\e(D)} \int_{2B_i\setminus B_i} |\vartheta_\e|^{3m_\vartheta}\bigg)^\frac23\bigg(\sum_{z_i\in\Phi^\e(D)} 1\bigg)^\frac13\\
&\quad+\int_{D_\e}\big|\nabla|\vartheta_\e|^\frac{m_\vartheta}{2}\big|^2+\int_{D_\e}|\vartheta|^{3m_\vartheta}\\
&\lesssim \big(\#\{z_i\in\Phi^\e(D)\}\big)^\frac13+1,
\end{align*}
where in the last inequality we used the uniform bounds on $\vartheta_\e$ and $\nabla|\vartheta_\e|^\frac{m_\vartheta}{2}$. From Remark \ref{rmk:SLLN}, for $\e>0$ small enough, the number of points $z_i\in\Phi^\e(D)$ is bounded by $C\e^{-3}$, which immediately implies our desired result.
\end{proof}

Summarizing all the above results, we know the existence of an almost surely positive random variable $\e_0(\omega)$ such that for all $0<\e\leq \e_0$ the solution $(\rho_\e,\vb u_\e,\vartheta_\e)$ to \eqref{eq:conteq}-\eqref{specificEnergy}, suitably extended to the whole of $D$, satisfies
\begin{align}\label{finalBounds}
\|\vb u_\e\|_{H^1_0(D)}+\|\rho_\e\|_{L^{\gamma+\Theta}(D)}+\|\vartheta_\e\|_{H^1(D)}+\|\vartheta\|_{L^{3m_\vartheta}(D)}\leq C,
\end{align}
where $\Theta$ is defined in \eqref{definTheta}. Further, $\vartheta_\e$ has a well defined trace on each $\d B_{\e^\alpha r_i}(\e z_i)$, the norm of which is controlled by Corollary \ref{cor:TraceBound}.

\section{Equations in fixed domain}\label{sec:fixedDom}

In this section, we will show the homogenization result for Navier-Stokes-Fourier equations in a randomly perforated domain in the subcritical case $\alpha>3$. The proof of such result in the case of periodically arranged holes is given in \cite[Section 5]{LuPokorny21}. Since their methods apply almost verbatim to our situation, we will just focus on the differences due to the random setting. Again, we will always assume that the moment bound $m_r\geq 3$ in \eqref{Conditionm} in order to control the measures of $D_\e$ and $\d D_\e$.

First, the bounds in \eqref{finalBounds} enable us to extract subsequences such that
\begin{align*}
&\vb u_\e\weak \vb u \text{ weakly in } H_0^1(D),\quad \vb u_\e\to \vb u \text{ strongly in } L^q(D) \text{ for all } 1\leq q<6,\\
&\rho_\e\weak \rho \text{ weakly in } L^{\gamma+\Theta}(D),\\
&\vartheta_e\weak \vartheta \text{ weakly in } H^1(D),\quad \vartheta_\e\to \vartheta \text{ strongly in } L^q(D) \text{ for all } 1\leq q<3m_\vartheta.
\end{align*}

To pass to the limit in the energy balance \eqref{eq:energyBalance}, we use its weak formulation \eqref{eq:weakEnergy} and the fact $\vb u_\e=0$ in $D\setminus D_\e$ to write
\begin{align}\label{eq:LimitEnergy}
\begin{split}
&-\int_D \bigg(\rho_\e E(\rho_\e,\vb u_\e,\vartheta_\e)\vb u_\e+p(\rho_\e,\vartheta_\e)\vb u_\e-\mathbb{S}(\vartheta_\e,\nabla\vb u_\e)\vb u_\e-\kappa(\vartheta_\e)\nabla\vartheta_\e\bigg)\cdot\nabla\psi\\
&\quad+L\int_{\d D} (\vartheta_\e-\vartheta_0)\psi-\int_D (\rho_\e\vb f+\vb g)\cdot\vb u_\e\psi\\
&=\int_{D\setminus D_\e}\kappa(\vartheta_\e)\nabla\vartheta_\e\cdot\nabla\psi-L\int_{\cup_{z_i\in\Phi^\e(D)} \d B_{\e^\alpha r_i}(\e z_i)}(\vartheta_\e-\vartheta_0)\psi\\
&=: I_1+I_2
\end{split}
\end{align}
for any $\psi\in C^1(\overline{D})$, where $E(\rho_\e,\vb u_\e,\vartheta_\e)$ is the total energy from \eqref{totalEnergy}. We want to show that both integrals on the right hand-side vanish as $\e\to 0$. For $I_1$, by Hölder's inequality, we get
\begin{align*}
|I_1|\leq C\|\nabla\psi\|_{L^\infty(D)}(1+\|\vartheta_\e\|_{L^{3m_\vartheta}(D\setminus D_\e)}^{m_\vartheta})\|\nabla\vartheta_\e\|_{L^2(D\setminus D_\e)}|D\setminus D_\e|^\frac16\to 0,
\end{align*}
where we used that $|D\setminus D_\e|\to 0$ by \eqref{measures}. For $I_2$, let us set $B_i:=B_{\e^\alpha r_i}(\e z_i)$. Using Corollary \ref{cor:TraceBound} and that $\|\vartheta_0\|_{L^q(\d D_\e)}$ is uniformly bounded for some $q>1$ w.r.t.~$\e$, together with $\alpha>3$, $m_r>2$ and Lemma \ref{lem:SLLN}, we obtain
\begin{align*}
|I_2|&\lesssim \|\vartheta_\e\|_{L^{2m_\vartheta}(\cup_{z_i\in \Phi^\e(D)}\d B_i)}\bigg|\bigcup_{z_i\in\Phi^\e(D)} \d B_i\bigg|^\frac{2m_\vartheta-1}{2m_\vartheta}+\|\vartheta_0\|_{L^q(\cup_{z_i\in \Phi^\e(D)}\d B_i)}\bigg|\bigcup_{z_i\in\Phi^\e(D)} \d B_i\bigg|^\frac{q-1}{q}\\
&\lesssim \e^{-\frac{1}{2m_\vartheta}}\bigg(\sum_{z_i\in\Phi^\e(D)} \e^{2\alpha} r_i^2\bigg)^\frac{2m_\vartheta-1}{2m_\vartheta}+\bigg(\sum_{z_i\in\Phi^\e(D)} \e^{2\alpha} r_i^2\bigg)^\frac{q-1}{q}\\
&\lesssim \e^\frac{(2\alpha-3)(2m_\vartheta-1)-1}{2m_\vartheta}+\e^\frac{(2\alpha-3)(q-1)}{q}\to 0,
\end{align*}
where we used that $(2\alpha-3)(2m_\vartheta-1)>1$ due to our assumptions $\alpha>\frac{3m_{\vartheta}-2}{m_\vartheta-2}$ and $m_\vartheta>2$.  Hence, letting $\e\to 0$ on the left hand-site of \eqref{eq:LimitEnergy}, we get by the strong convergences of $\vb u_\e$ and $\vartheta_\e$
\begin{align*}
&-\int_D \bigg(\big(\overline{\rho e(\vartheta,\rho)}+\frac12 \rho |\vb u|^2+\overline{p(\vartheta,\rho)}-\mathbb{S}(\vartheta,\nabla\vb u)\big)\vb u-\kappa(\vartheta)\nabla\vartheta\bigg)\cdot\nabla\psi\\
&\quad+L\int_{\d D} (\vartheta-\vartheta_0)\psi=\int_D (\rho\vb f+\vb g)\cdot \vb u\psi.
\end{align*}
Here, $\overline{f(\vartheta,\rho)}$ denotes the weak limit of a function $f(\vartheta_\e,\rho_\e)$ in some suitable $L^q$-space. Further, the temperature $\vartheta>0$ a.e.~in $D$, which can be proven as shown in \cite[Lemma 5.1]{LuPokorny21} when we replace the estimate $\big|\cup_{n=1}^{N(\e_l)} T_{n,\e_l}\big|\lesssim l^{-3(\alpha-1)}$ therein by \eqref{measures}$_2$ for the specific sequence $\e_l=l^{-1}$.

It remains to show the energy balance for the limit functions, which is in fact a consequence of the strong convergence of the density $\rho_\e$ to $\rho$ at least in $L^1(D)$. In fact, the strong convergence holds in $L^q(D)$ for any $1\leq q<\gamma+\Theta$. Since the proof of this fact is nowadays well understood and applies verbatim to our case of a random perforation, we refer to \cite[Section 5.3]{LuPokorny21}. We note that equation (5.23) in there should read $\gamma+\Theta\geq \gamma+1>3$.

We now turn to the continuity and momentum equation. Recall that the continuity equation holds in the weak and renormalized sense \eqref{eq:weakContinuity} and \eqref{eq:renormalized}, so we obtain by the strong convergence of $\vb u_\e$ to $\vb u$
\begin{align}\label{eq:limitCont}
\div(\rho \vb u)=0 \text{ in } \mathcal{D}^\prime(\R^3)
\end{align}
and
\begin{align*}
\div(\overline{b(\rho)}\vb u)+\overline{(\rho b^\prime(\rho)-b(\rho))\div(\vb u)}=0 \text{ in } \mathcal{D}^\prime(\R^3),
\end{align*}
where we denote by $\overline{f(\rho)}$ the weak limit of a function $f(\rho_\e)$ in some suitable $L^q$-space. Moreover, by Remark \ref{rem:renormalized}, \eqref{eq:limitCont} implies that the couple $(\rho,\vb u)$ fulfils the renormalized continuity equation \eqref{eq:renormalized} for any $b\in C([0,\infty))\cap C^1((0,\infty))$ satisfying the conditions of Remark \ref{rem:renormalized}.

To pass to the limit in the momentum equation, we need to construct suitable test functions. To this end, we recall a lemma from \cite{BellaOschmann2021a}:

\begin{lemma}\label{lm:cutoff}
Let $\alpha > 2$, $D\subset\R^3$ be a bounded $C^2$ domain with $0 \in D$, and $(\Phi,\mathcal{R})=(\{z_i\},\{r_i\})$ be a marked Poisson point process with intensity $\lambda > 0$ and $r_i \ge 0$ with $\mathbb{E}(r_i^{m_r}) < \infty$ for ${m_r>\max\{3/(\alpha-2),3\}}$. Then for any $1 < q < 3$ such that $(3-q)\alpha - 3 > 0$ and for almost every $\omega$ there exist a positive $\e_0(\omega)$ and a family of functions $\{ g_\eps\}_{\eps > 0} \subset W^{1,q}(D)$ such that for $0 < \e \leq \e_0$,
 \begin{equation}
  g_\e = 0 \quad \textrm{ in } \bigcup_{z_j \in \Phi^\eps(D)} B_{\eps^\alpha r_j}(\eps z_j), \qquad g_\e \to 1 \quad \textrm{ in } W^{1,q}(D) \textrm{ as } \eps \to 0
 \end{equation}
 and there is a constant $C>0$ such that
 \begin{equation}
  \| g_\e - 1 \|_{W^{1,q}(D)} \le C\e^{\sigma} \qquad \textrm{ with } \sigma := ((3-q)\alpha-3)/q. 
 \end{equation}
\end{lemma}

\begin{proof}
By $m_r>3/(\alpha-2)$ and Lemma \ref{lem:MainProp}, all the balls $\{B_{2\e^\alpha r_j}(\e z_j)\}_{z_j\in\Phi^\e(D)}$ are disjoint. Thus, there exist functions $g_\e\in C^\infty(D)$ such that
\begin{align*}
&0\leq g_\e\leq 1,\quad g_\e = 0 \textrm{ in } \bigcup_{z_j \in \Phi^\eps(D)} B_{\eps^\alpha r_j}(\eps z_j),\quad g_\e = 1 \textrm{ in } D\setminus\bigcup_{z_j \in \Phi^\eps(D)} B_{2\eps^\alpha r_j}(\eps z_j),\\
&\|\nabla g_\e\|_{L^\infty(B_{2\e^\alpha r_j}(\e z_j))}\leq C (\e^\alpha r_j)^{-1} \textrm{ for all } z_j\in\Phi^\e(D),
\end{align*}
where the constant $C>0$ is independent of $\e$ and $r_j$. Moreover, since $m_r\geq 3$, \eqref{eq:ergodic} yields 
$\lim\limits_{\eps \to 0} \eps^3 \sum_{z_j \in \Phi^\e(D)} r_j^3 = C$, thus implying
\begin{equation}\nonumber
 \biggl| \bigcup_{z_j \in \Phi^\e(D)} B_{2\eps^\alpha r_j}(\eps z_j) \biggr| \le |B_2| \eps^{3\alpha} \sum_{z_j \in \Phi^\e(D)} r_j^3 \le C \eps^{3(\alpha-1)}
\end{equation}
for $\eps > 0$ small enough. This together with direct calculation yields that for any $1<q<3$,
\begin{align*}
\|1-g_\e\|_{L^q(D)}\leq C\e^\frac{3(\alpha-1)}{q},\quad \|\nabla g_\e\|_{L^q(D)}\leq C\e^{\frac{(3-q)\alpha-3}{q}},
\end{align*}
which finally leads to
\begin{align*}
\|1-g_\e\|_{W^{1,q}(D)}=\|1-g_\e\|_{L^q(D)}+\|\nabla g_\e\|_{L^q(D)}\leq C\e^\sigma.
\end{align*}
\end{proof}

Using the cut-off functions from Lemma \ref{lm:cutoff}, the proof of Lemma 5.2 in \cite{LuPokorny21} applies verbatim to our situation, yielding the following result:
\begin{lemma}
Under the assumptions of Theorem \ref{thm:main}, there holds
\begin{align*}
\div(\rho_\e\vb u_\e\otimes\vb u_\e)+\nabla p(\vartheta_\e,\rho_\e)-\div\mathbb{S}(\vartheta,\nabla\vb u_\e)=\rho_\e\vb f+\vb g+F_\e \text{ in } \mathcal{D}^\prime(D),
\end{align*}
where $F_\e$ is a distribution satisfying
\begin{align*}
|\langle F_\e,\phi\rangle_{\mathcal{D}^\prime(D),\mathcal{D}(D)}|\leq C\e^\sigma\big(\|\nabla\phi\|_{L^{\frac{3(\gamma+\Theta)}{2(\gamma+\Theta)-3}+\xi}(D)}+\|\phi\|_{L^{r_1}(D)}\big)
\end{align*}
for all $\phi\in \mathcal{D}(D)$, where $\Theta$ is defined in \eqref{definTheta} and $\sigma,\xi,r_1$ are defined such that the following conditions are fulfiled:
\begin{align*}
&0<\xi<1,\quad 0<h(\xi):=3(\alpha-1)\bigg(\frac{3(\gamma+\Theta)}{2(\gamma+\Theta)-3}+\xi\bigg)^{-1}-\alpha,\\
&1<r_1<\infty,\quad \frac{1}{r_1}+\bigg(\frac{3(\gamma+\Theta)}{2(\gamma+\Theta)-3}+\xi\bigg)^{-1}=\frac{2(\gamma+\Theta)-3}{3(\gamma+\Theta)},\\
&0<\sigma<\infty,\quad \sigma:=\min\left\lbrace\frac{3(\alpha-1)}{r_1},h(\xi)\right\rbrace.
\end{align*}
\end{lemma}

To finish the proof of Theorem \ref{thm:main}, we have to show
\begin{align*}
\overline{\rho e(\vartheta,\rho)}=\rho e(\vartheta,\rho),\quad \overline{p(\vartheta,\rho)}=p(\vartheta,\rho).
\end{align*}
By the strong convergence of $\vartheta_\e$ to $\vartheta$ in any $L^q(D)$ for $1\leq q<3m_\vartheta$, it is sufficient to show the strong convergence of $\rho_\e$ to $\rho$, which is proven in \cite[Section 5.3]{LuPokorny21}. To summarize, we finally have that the weak limit $(\rho,\vb u,\vartheta)$ is a solution to problem \eqref{eq:conteq}--\eqref{specificEnergy} in the limit domain $D$. This completes the proof of Theorem \ref{thm:main}.\\
  
\noindent
{\bf Acknowledgement.} The author thanks Peter Bella for helpful discussions on the problem. The author was partially supported by the German Science Foundation DFG in context of
the Emmy Noether Junior Research Group BE 5922/1-1.  
  
\phantomsection 

 \bibliographystyle{amsplain}
 \bibliography{Lit}

\end{document}